\documentclass[a4paper]{ifac}
\usepackage{natbib}        
\usepackage{amsmath}
\usepackage{amsfonts}
\usepackage{tikz}
\usepackage{paralist}
\usepackage{graphicx}
\usepackage{comment}
\usepackage{amssymb}
\usepackage{algorithm}
\usepackage{algpseudocode}
\algrenewcommand\textproc{}

\newtheorem{assumption}{\textbf{Assumption}}
\newtheorem{problem}{\textbf{Problem}}
\newtheorem{proposition}{\textbf{Proposition}}

\newtheorem{lemma}{\textbf{Lemma}}
\newtheorem{proof}{\textbf{Proof}}
\newtheorem{remark}{\textbf{Remark}}
\newtheorem{corollary}{\textbf{Corollary}}
\begin{document}
\begin{frontmatter}

\title{Mixed $\mathcal{H}_2/\mathcal{H}_\infty$-Policy Learning Synthesis} 

\thanks[footnoteinfo]{The author is with Microsoft Research, NYC.}

\author{Lekan Molu}

\address{Microsoft Research, 300 Lafayette Street, New York City, NY 10012. \\  lekanmolu@microsoft.com.}

\begin{abstract}                
		A robustly stabilizing optimal control policy in a model-free mixed $\mathcal{H}_2/\mathcal{H}_\infty$-control setting is here put forward for counterbalancing the slow convergence and non-robustness of traditional high-variance policy optimization (and by extension policy gradient) algorithms. Leveraging It\^{o}'s stochastic differential calculus, we iteratively solve the system's continuous-time closed-loop generalized algebraic Riccati equation whilst updating its admissible controllers in a two-player, zero-sum differential game setting. Our new results  are illustrated by learning-enabled control systems which gather previously disseminated results in this field in one holistic data-driven presentation with greater simplification, improvement, and  clarity. 
\end{abstract}


\end{frontmatter}
\definecolor{light-blue}{rgb}{0.30,0.35,1}
\definecolor{light-green}{rgb}{0.20,0.49,.85}
\definecolor{purple}{rgb}{0.70,0.69,.2}

\newcommand{\lb}[1]{\textcolor{light-blue}{#1}}
\newcommand{\bl}[1]{\textcolor{blue}{#1}}

\newcommand{\maybe}[1]{\textcolor{gray}{\textbf{MAYBE: }{#1}}}
\newcommand{\inspect}[1]{\textcolor{cyan}{\textbf{CHECK THIS: }{#1}}}
\newcommand{\more}[1]{\textcolor{red}{\textbf{MORE: }{#1}}}

\newcommand{\subfigureautorefname}{\figureautorefname}

\newcommand{\cmt}[1]{{\footnotesize\textcolor{red}{#1}}}
\newcommand{\todo}[1]{\textcolor{cyan}{TO-DO: #1}}
\newcommand{\review}[1]{\noindent\textcolor{red}{$\rightarrow$ #1}}
\newcommand{\response}[1]{\noindent{#1}}
\newcommand{\stopped}[1]{\color{red}STOPPED HERE #1\hrulefill}

\newcounter{mnote}
\newcommand{\marginote}[1]{\addtocounter{mnote}{1}\marginpar{\themnote. \scriptsize #1}}
\setcounter{mnote}{0}
\newcommand{\ie}{i.e.\ }
\newcommand{\eg}{e.g.,\ }
\newcommand{\cf}{c.f.\ }
\newcommand{\yes}{\checkmark}
\newcommand{\no}{\ding{55}}

\newcommand{\flabel}[1]{\label{fig:#1}}
\newcommand{\seclabel}[1]{\label{sec:#1}}
\newcommand{\tlabel}[1]{\label{tab:#1}}
\newcommand{\elabel}[1]{\label{eq:#1}}
\newcommand{\alabel}[1]{\label{alg:#1}}
\newcommand{\fref}[1]{\cref{fig:#1}}
\newcommand{\sref}[1]{\cref{sec:#1}}
\newcommand{\tref}[1]{\cref{tab:#1}}
\newcommand{\eref}[1]{\cref{eq:#1}}
\newcommand{\aref}[1]{\cref{alg:#1}}

\newcommand{\bull}[1]{$\bullet$ #1}
\newcommand{\argmax}{\text{argmax}}
\newcommand{\argmin}{\text{argmin}}
\newcommand{\mc}[1]{\mathcal{#1}}
\newcommand{\bb}[1]{\mathbb{#1}}

\def\tidx{t}
\newcommand{\Note}[1]{}
\renewcommand{\Note}[1]{\hl{[#1]}}  


\def\kau{\mc{K}}
\def\particle{\bm{x}}
\def\materialresponse{\bm{G}}
\def\orthoggroup{{\textit{SO}}(3)}
\def\liegroup{{\textit{SE}}(3)}
\def\liealgebra{\mathfrak{se}(3)}
\def\identity{\bm{I}}
\newcommand{\trace}[1]{\textbf{tr}(#1)}

\def\rot{{R}}
\def\flock{F}
\def\rthree{\bb{R}^3}
\def\reline{\bb{R}}
\def\targetset{\mathcal{L}}
\def\traj{\xi}
\def\ren{\bb{R}^n}
\def\skew{S}
\def\jacob{J}
\def\state{\bm{x}}
\def\statex{x}
\def\statey{y}
\def\statez{z}
\def\hot{h.o.t.\ }
\def\lhs{l.h.s.\ }
\def\rhs{r.h.s.\ }
\def\hinf{\mc{H}_\infty}
\def\htwo{\mc{H}_2}
\def\identity{I}
\def\costdiff{\mathbf{\tilde{V}}}
\def\pursuer{\bm{P}}
\def\evader{\bm{E}}
\def\gain{\bm{k}}
\def\control{\bm{u}}
\def\disturb{\bm{v}}
\def\switchcurve{\bm{\gamma}}
\def\valuefunc{\bm{V}}
\def\valueterm{\bm{g}}
\def\lpspace{L^2({\mc{S}}; \mc{F})}
\def\lpdual{L^2({\mc{S}}; \breve{\mc{F}})}
\def\valuetensor{\mathds{V}}
\def\wtensor{\mathds{W}}
\def\valuecore{\mathds{V}^c}
\def\hamfunc{\bm{H}}
\def\hamtensor{\mathds{H}}
\def\uppervalue{\bm{V}^+}
\def\lowervalue{\bm{V}^-}
\def\upperham{\bm{H}^+}
\def\lowerham{\bm{H}^-}
\def\hilbertparam{\bm{\phi}}
\def\hilbertcoeff{\bm{\psi}}
\def\hilbertparamspace{\bm{\Phi}}
\def\hilbertcoeffspace{\bm{\Psi}}
\def\hilbertspace{\mathcal{F}}
\def\hilbertdual{\mathcal{S}}
\def\reducedbasis{\Xi_r}
\def\basis{\mathbf{e}}
\def\openset{\Omega}
\def\partition{\Gamma}
\def\spatialdomain{\Omega}
\def\timeinterval{I}
\def\liederi{L}
\def\stspace{\mc{X}}
\def\eigmin{\underline{\sigma}}
\def\eigmax{\sigma_{max}}
\def\eigvamin{\underline{\lambda}}
\def\eigvamax{\bar{\lambda}}

\def\vec{\texttt{vec}}
\def\svec{\texttt{svec}}
\def\vecv{\texttt{vecv}}
\def\mat{\texttt{mat}}
\def\smat{\texttt{smat}}
\section{Introduction}

We consider system stabilization together with Zames' sensitivity compensation in plants disturbed by additive Wiener process and uncertainties~\citep{Zames1981} under model-free policy optimization and gradient settings. We pose our solution in a mixed $\htwo/\hinf$ linear quadratic (LQ) optimal control problem (OCP)~\citep{Khargonekar1988} within the family of policy optimization (PO) schemes. Connecting this mixed design synthesis to modern policy optimization algorithms in machine learning, we optimize a performance index that is the upper bound on the $\htwo$-norm of the plant transfer function subject to the plant's $\hinf$-norm constraints: we must find feasible stabilizing policies whilst guaranteeing robustness to a measure of disturbance~\citep{Zhang2021, CuiMoluxTAC}. 

%

PO algorithms, which encapsulate {policy gradient} (PG) methods ~\citep{ShamNPG, PolGradSham}, are attractive for modern data-driven problems since they
\begin{inparaenum}[(i)]
	\item admit continuously differentiable policy parameterization;
	\item  are easily extensible to function approximation settings; and
	\item admit structured state and control spaces.
\end{inparaenum}
As such, PG algorithms are increasingly becoming integral to modern engineering solutions, recommender systems, finance, and critical infrastructure given the growing complexity of the systems that we build and the massive availability of datasets. 
A major drawback of PG algorithms, however, is that they compute {high-variance} gradient estimates of the LQR costs from Monte-Carlo {trajectory rollouts} and {bootstrapping}. As such, they tend to possess slow convergence guarantees.

To address PG's characteristic non-robustness to uncertainty,  and its characteristic slow convergence, recent efforts have proposed mixed $\htwo/\hinf$ control proposal~\citep{Zhang2019, Zhang2021, CuiMoluxTAC} as a risk-mitigation design tool: imposing an additional $\hinf$-norm constraint on the  $\htwo$ cost to be minimized, one guarantees {robust stability and performance} in the presence of unforeseen uncertainties, noise, worst-case disturbance or incorrectly estimated dynamics -- signatures of PG algorithms.


Under {stabilizable} and {observable} system parameter conditions, ~\citep{Zhang2021} established {globally sub-linear} and {locally super-linear} convergence rates in linear quadratic (LQ) zero-sum dynamic game settings. We improved upon these convergence rates in~\citep{CuiMoluxTAC} by solving the PO problem recursively {given an initial stabilizing feedback gain that also preserved the $\hinf$ robustness metric}. In many modern engineering systems that employ PO, however, stochastic system parameters often have to be identified from  nonlinear system trajectory data. For these schemes to work, the control designer may need to linearize nonlinear trajectories about successive equilibrium points (whilst imposing the standard stabilizability and observability constraints on system parameters to be identified). In this paper, we take steps to curb our earlier stabilizability and observability assumptions in ~\citep{CuiMoluxTAC}.

\noindent \textbf{Contributions}: We here present a holistic synthesis of our previous dissemination, initiate a search  for the initial $\hinf$-norm constraints-preserving feedback gain, $K_1$, and demonstrate the efficacy of our results on a nonlinear numerical experimental setting. The rest of this paper is structured as follows: in \S \ref{sec:backs}, we introduce notations and contextualize the problem; in \S \ref{sec:methods} we present our methods; results that back up our claims are set forth in \S \ref{sec:results}. We draw conclusions in \S \ref{sec:conclude}.


\section{Preliminaries}
\label{sec:backs}

%
\subsection{Notations}
\label{sec:notations}

We adopt standard vector-matrix notations throughout.  Conventions: Capital and lower-case Roman letters are respectively matrices and vectors; calligraphic letters are sets. Exceptions: time variables \eg $t, t_0, t_f, T$ will always be real numbers.

The $n$-dimensional Euclidean space is  $\reline^n$. The real and imaginary parts of the complex $s$-plane are respectively $\texttt{Re}(s)$ and $\texttt{Imag}(s)$. The singular values of $A \in \reline^{n\times n}$ are  $\sigma_i(A), i=1,\cdots,n$. The standard $\hinf$ norm of a  complex matrix-valued function $G(j\omega)$ is defined over the analytic and bounded functions in the open right-half plane as $\|G(j\omega)\|_\infty = \sup_{\omega \in \mathbb{R}} \eigmax(G(j\omega))$ where $\eigmax(\cdot)$ denotes the maximum singular value. The $\mc{L}_2$ norm for a signal, function, or the induced matrix norm is  denoted  $\|\cdot\|_2$. We let $\{\lambda_i(X)\}_{i=1}^n$ denote the $n$-eigenvalues of $X \in \reline^{n\times n}$ where $\lambda_1 < \lambda_2 < \cdots <\lambda_n$. When an optimized variable \eg $u$ is optimal with respect to an index of performance, it shall be denoted $u^\star$.

All vectors are column-stacked. The Kronecker product of  $A \in
\reline^{m \times n}$ and $B \in \reline^{p \times q}$ is $A \otimes B$. Symmetric $n$-dimensional matrices shall belong in $\mathbb{S}^n$. A  positive definite (resp. negative definite) $A$ is written $A \succ 0$ (resp. $A \prec 0$).  We denote the index of a given matrix or vector by subscripts. Colon notation denotes the full range of a given index. Indexing ranges over $1, \cdots, n$ for an $n$-dimensional vector. The $i$th row of a given matrix $A$ is $A_{[i,:]}$, while the $j$th column of $A$ is $A_{[:,j}]$. Blockwise indexing follows similar conventions.

Denote by $x_{ij}$ the $(ij)$'th entry of $X \in \reline^{m\times n}$ and by $x_{i}$ the $i$'th element of $x \in \reline^{n}$. The full vectorization of $X  \in \reline^{m\times n}$ is the $mn\times 1$ vector obtained by stacking the columns of $X$  on top of one another \ie $\vec(X) = \left[x_{11}, x_{21}, \cdots, x_{m1}, x_{12}, \cdots, x_{m2}, \cdots, x_{mn}\right]^T$. Let $P \in \bb{S}^n$, then the half-vectorization of $P$ is the $n(n+1)/2$ column vector as a result of a vectorization of upper-triangular part of $P$ i.e.  $\svec(P)=\left[p_{11}, p_{12}, \cdots, p_{1n}, \cdots, p_{nn}\right]^T$. The vectorization of the dot product $\langle x, x^T\rangle $, where $x
 \in \ren$,  is  $\vecv(x) := [x_1^2, \cdots, x_1x_n, x_2 x_1, x_2^2, x_2x_3,\cdots,x_n^2]^T$. The inverse of $\vec(x)$ and $\svec(y)$  are respectively the full and symmetric matricizations: $ \mat_{m\times n}(x)=\left(\vec(I_n)^T \otimes I_m\right) \linebreak[0]\left(I_n \otimes x\right)$, and $\smat_m(P)$ so that $ \smat(\svec(p))=P$. Here, $x \in \reline^{mn}$ and $y \in \reline^{m(m+1)/2}$ for $n,m \in \reline_{\ge 0}$. Finally, we denote by $T_{\vec}(A)$ the vectorization of $A^T$ \ie $\vec(A^T) = T_{\vec}(\vec(A))$.

\subsection{System Description}
\label{subsec:sys}
Consider the following nonlinear system
\begin{subequations}
	\begin{align}
		\dot{x}(t) &= f(x; t) + g(x)u(t) + h(x)w(t), \, x(0) = x_0 \\
		z(t) &= \mc{G}(x, u; t), \,\, z(0) = z_0
	\end{align}
	\label{eq:nlnr}
\end{subequations}
where $f(\cdot), \, g(\cdot), \, h(\cdot)$, and $\mc{G}(\cdot)$ are nonlinear functions with appropriate dimensions. The state process is $x\in \ren$, the controlled output process is  $z\in \reline^m$, the control input is $u \in \mc{U} \subseteq \reline^{\mathfrak{p}}$, and the vector-valued (stochastic) Wiener process is $w \in \mc{W} \subseteq \reline^{\mathfrak{q}}$. Let the following finite-dimensional linear time-invariant (FDLTI) system describe the resulting linearized stochastic differential equation
\begin{subequations}
	\begin{align}
		dx(t) &= A x(t) dt + B_1 u(t) dt + B_2 dw(t), \,\,x(0) = x_0 \label{eq:leqg} \\
		z(t) &= C x(t) + D u(t), \quad  z(0)=0,
	\end{align}
	\label{eq:system}
\end{subequations}
where $dw$ is the Gaussian white noise; $x(0)$ is an arbitrary zero-mean Gaussian random vector  independent of $w(t)$; and $A, \, B_1, \, B_2, \, C, \, D$ are real matrix-valued functions of appropriate dimensions.  The random signal, $x(0)$, and process $w(t)$ are defined over a complete probability space $(\Omega, \mc{F}, \mc{P})$ where $\Omega$ is $w$'s sample space, $\mc{F}$ is the $\sigma$-algebra \ie the filtration generated by $w$, and $\mc{P}$ is the probability measure on which $w(t)$ is drawn for a $t \in [0, T]$ (where $T>0$ is fixed).

\begin{assumption}
	We impose the following conditions on the algorithm to be presented. We take $C^TC \triangleq Q\succ 0$, $D^T\left(C, \, D\right) = \left(0, \, R \right)$ for some $R\succ 0$; and should one wish that the noise process in \eqref{eq:system} be statistically independent, then we may take $B_2 B_2^T$ = 0. Seeing we are seeking a linear feedback controller for \eqref{eq:system}, we require that the pair $(A, B_1)$ be \textit{stabilizable}. We expect to compute solutions via an optimization process, therefore we require that \textit{unstable modes of $A$ must be observable through $Q$}. Whence $(\sqrt{Q}, A)$ must be \textit{detectable}.
	\label{ass:realizability}
\end{assumption}

\begin{problem}[Problem Statement]
	The \textit{goal is to keep the controlled process, $z$, small in an infinite-horizon LTI constrained optimization setting under a minimizing control $u$ in spite of unforeseen disturbances $w$}. 
	%
\end{problem}

Let the closed-loop operator (under an arbitrary negative feedback gain $K \in \mc{K}, u(x(t)) = -K x(t)$) mapping $w$ to $z$ be $\|T_{zw}(K)\|_2$. Then, 
\begin{align}
	T_{zw}(K) = \left(C-DK\right)(s\identity-A+B_1 K)^{-1} B_2.
	\label{eq:tranfer_cl}
\end{align}
Or in \citep{ZhouEssentials}'s packed representation, 
\begin{align}
	T_{zw}(K) \triangleq \left[\begin{array}{c|c}
		A - B_1 K & B_2 \\
		\hline \\
		C-DK & 0
	\end{array}\right]. 
	\label{eq:packed}
\end{align}

Design principles in linear control theory exist for solving problem \eqref{eq:system} when the covariance of the noise model has a small magnitude. For stochastic $\htwo$ control problems with large noise intensities (such as PG methods), it suffices to solve a \textit{linear exponential quadratic control problem} under a robustness constraint. To further contextualize the problem, let us formally introduce the design problem.

\subsection{Risk-Sensitive LEQG as a Mixed Design Problem}
In ~\citep{ZhangLongForm}, the authors established  that the risk-sensitive infinite-horizon linear exponential quadratic Gaussian (LEQG) state-feedback control problem~\citep{Jacobson1973, Whittle1981} is an equivalent mixed-$\htwo/\hinf$ control design  problem for \textit{linear time-invariant} systems with additive noise of the form \eqref{eq:system}. We iterate upon this contribution  since it introduces a measure of risk-design as an implicit robustness metric when the process noise has a large covariance intensity. And this is typical for policy gradient settings. The state evolves according to \eqref{eq:leqg} and without loss of generality, the stochastic linear system's performance criterion is
\begin{align}
	\mc{J}(K) &= \limsup_{t_f\rightarrow \infty}\dfrac{2 \gamma^2}{t_f}  \log \bb{E} \text{ exp }\left[\dfrac{1}{2 \gamma^2}\int_{t=0}^{t_f}\langle z(t), z(t) \rangle dt \right].
	\label{eq:leqg_perf}
\end{align}
Suppose that the variance term $\gamma^{-2} \texttt{var}(z^Tz)$ is small, then $\gamma$ is a measure of \textit{risk-propensity} if $\gamma >0$; similarly, $\gamma$ can be considered as a measure of \textit{risk-aversion} if $\gamma < 0$; and $\gamma$ is a measure of \textit{risk-neutrality} if $\gamma=0$ (equivalent to the standard state-feedback LQP). Given LEQG's connection under risk-propensity to the high-variance associated with PG algorithms, throughout the rest of this paper we take $\gamma>0$ in our optimization process.

\subsection{Mixed $\htwo/\hinf$-Policy Optimization Synthesis} We now define the standard mixed $\htwo/\hinf$ control problem: given system \eqref{eq:system} and a real number $\gamma>0$, find an \textit{admissible controller} $K$ that exponentially\footnote{Concerning matters relating to linear systems, we take exponential stability to mean internal stability, so that the transfer matrix belongs in the real-rational $\hinf$ space \ie $T_{zw} \in \mc{RH}_\infty$.}
stabilizes \eqref{eq:tranfer_cl} and renders  $\|T_{zw}\|_\infty < \gamma$. The set of all \textit{suboptimal} controllers that robustly stabilizes \eqref{eq:system} against all (finite gain) stable perturbations $\Delta$, interconnected to the system by $w = \Delta z$, such that $\|\Delta\|_\infty \le 1/\gamma$ can be succinctly denoted as
\begin{align}
	\mc{K} = \{\,K: \,\,\lambda_i(A-B_1K) < 0,  \,\, \|T_{zw}(K)\|_\infty < \gamma \}
	\label{eq:constraint}
\end{align}
\text{ for } $i=1, \cdots, n$.
We say $\mc{K} \neq \emptyset$ if the pair $(A, B_1)$ is stabilizable and $(C, A)$ is detectable \cf \eqref{eq:system}.

Aside from the constraint \eqref{eq:constraint}, the mixed $\htwo/\hinf$ performance measure can be framed as minimizing an ``upper-bound" on the $\htwo$-norm of the cost subject to the constraint $\|T_{zw}\|_\infty < \gamma$~\citep{Bernstein1989} for a $\gamma>0$. Abusing notation, let $\mc{J}(T_{zw})$ denote the (closed-loop) mixed $\htwo/\hinf$-control performance measure for the LTI system \eqref{eq:system}.
\begin{lemma}
	The problem \eqref{eq:system} with (a slightly abused) quadratic performance measure \eqref{eq:leqg_perf} \ie
	\begin{align}
		\mc{J}(T_{zw}) &=  \bb{E}\left\{\limsup_{t_f \rightarrow \infty} \left[\text{ exp }\left(2 \gamma^{-2}\int_{t=0}^{t_f}\langle z(t), z(t) \rangle dt \right) \right]\right\}
		\label{eq:leqg_perf_abused}
	\end{align}
	 admits a unique solution $x(t)$ to \eqref{eq:system} after optimizing  $\min_{K\in  \mc{K}}\mc{J}(\cdot)$ under the \textit{unique and optimal controller},
	\begin{align}
		u^\star(x(t)) = -R^{-1} B_1^T P(t) x(t), \quad t \in [0, \,t_f].
		\label{eq:duncan_control}
	\end{align}
	In \eqref{eq:duncan_control}, $P(t)$ is the unique, symmetric positive solution to the continuous-time (closed-loop) generalized algebraic Riccati equation (GARE)
	\begin{align}
	PA + A^TP - P\left(B_1 R^{-1}B_1^T -\gamma^{-2}B_2B_2^T\right)P + Q=0
	\label{eq:closed_loop_riccati}
	\end{align}
	if $Q \succ 0$ for a $\gamma >0$.
\end{lemma}
\begin{proof}
	This Lemma is the infinite-horizon retrofitting of Duncan's solution to the LEQG control value function based on a standard completion of squares and a Radon-Nikodym derivative~\citep[Th II.1]{Duncan2013}.
\end{proof}
\begin{corollary}[Th 9.7 \cite{BasarBook}]
	The GARE \eqref{eq:closed_loop_riccati} in an infinite-horizon LTI setting admits an equivalent \textit{LQ two-player zero-sum differential game} with the following \textit{upper value} 
	\begin{align}
		\begin{split}
			\mc{J}(T_{zw}) =  \limsup_{t_f \rightarrow \infty} \inf_{u \in \mc{U}} \sup_{w \in \mc{W}}  \int_{t=0}^{t_f}\left[x^T(t)Qx(t) +  \right. \\ \left.  u(\cdot)^TR u(\cdot)-\gamma^{-2} w^T(t)w(t)\right] dt, \,\, \forall \, x \in \ren
		\end{split}
		\label{eq:lq_two_player}
	\end{align}
	subject to assumption \ref{ass:realizability}~\citep[\S. 9.7]{BasarBook}. Note that $\gamma >0$ can be interpreted as an upper bound on the $L_2$ gain disturbance attenuation or the $\hinf$-norm of the system. In addition, let a finite scalar $\gamma^\infty >0$ exist, then for all $\Gamma \triangleq \inf \{\gamma > \gamma^\infty\}$, \eqref{eq:closed_loop_riccati} has a unique, finite, and positive definite solution if $(C, A)$ is observable.
\end{corollary}
\begin{corollary}[Th 4.8,~\citep{BasarBook}]
	If $\Gamma \neq \emptyset$, and if the LQ zero-sum differential game has a closed-loop perfect-state information structure defined on $[0, t_f], \,t_f \rightarrow \infty$, then \eqref{eq:lq_two_player} \textit{admits a unique solution} with feedback controls
	\begin{align}
		u^\star(t) = -R^{-1} B_1^T P_\gamma x(t), \,\, w^\star(t) = \gamma^{-2} B_2^T P_\gamma x(t)
		\label{eq:disturb_control}
	\end{align}
	for a $t \ge 0, \gamma > \gamma^\infty$. Note that $P_\gamma$ is the unique solution to \eqref{eq:closed_loop_riccati} in the class of positive definite feedback matrices (where the subscript $\gamma$ on $P$ denotes its direct dependence on $\gamma$) which makes the following feedback matrix Hurwitz,
	\begin{align}
		A_\gamma = A - (B_1R^{-1}B_1^T - \gamma^{-2}B_2B_2^T)P_\gamma.
		\label{eq:AHurwitz}
	\end{align}
	\label{cor:gamma_inf}
\end{corollary}
\begin{remark}
	Clearly, \eqref{eq:lq_two_player} is a minimax problem whose controller admits the form
	\begin{align}
		\begin{split}
		\min_{u \in \mc{U}} \max_{w \in \mc{W}}	\mc{J}(T_{zw}) =  \limsup_{t_f \rightarrow \infty}   \int_{t=0}^{t_f}\left[x^T(t)Qx(t)  \right. \\ \left.  +u^T(\cdot) \, R \, u(\cdot)-\gamma^{-2} w^T(t)w(t)\right] dt, \quad \forall \, x \in \ren.
		\end{split}
		\label{eq:lq_minimax}
	\end{align}
	Another common form of $\mc{J}(T_{zw})$ easily amenable to policy gradient algorithms is $J(T_{zw}) = \text{Tr}(P_\gamma B_2 B_2^T)$~\citep{Mustafa1989}.
\end{remark}

\begin{remark}
	The cost \eqref{eq:lq_minimax} is nonconvex and not coercive~\citep{Zhang2021}. However, our iterative solver~\citep{CuiMoluxTAC} 
	guarantees uniform linear convergence of the iterates during optimization.
\end{remark}

\begin{remark}
	The objective \eqref{eq:lq_minimax} is differentiable for any $K \in \mc{K}$, and its policy gradient $\nabla\mc{J}(T_{zw}):= 2(RK-B^TP_\gamma)\Lambda_\gamma \linebreak[0]$; here, $\Lambda_\gamma$ admits a form amenable to a (continuous time) closed-loop Lyapunov equation~\citep[Lemma A.4]{Zhang2021} i.e.,
	\begin{align}
		\begin{split}
			\Lambda_\gamma(A - B_1 K + \gamma^{-2} B_2 B_2^T P_\gamma)^T + \\
			(A - B_1 K + \gamma^{-2} B_2 B_2^T P_\gamma)\Lambda_\gamma + B_2 B_2^T = 0.
		\end{split}
	\end{align}
\end{remark}

For $\gamma>0$ and $\gamma \neq \sigma_i(B_2), i=1,\cdots,n$, we define the following closed-loop Hamiltonian matrix for a $(\gamma, K)$ pair
%
\begin{align}
	H(\gamma, K) 
	& = \begin{bmatrix}
		A - B_1 K & -\gamma^{-1} B_2 B_2^T \\
		-\gamma^{-1}(C^TC + K^TRK) & -(A-B_1 K)^T
	\end{bmatrix}
	\label{eq:ham_cl}
\end{align}
where we have used $R = (D^TD - \gamma^2I)$ and $S = (DD^T - \gamma^2I)$ as in~\citep[Eq. 2.2]{HinfFastCompute}.

\section{Methods}
\label{sec:methods}

We now introduce a nonlinear identification procedure, followed by linearization, and closed-loop $\hinf$ parameter search schemes. We close the section with an iterative solver for the GARE \eqref{eq:closed_loop_riccati}. 

\subsection{Nonlinear Identification and Linearization}
\label{subsec:narmax}
We remark that the user is not limited to the method to be introduced but in our experience, our identification scheme is interpretable and useful for debugging real-world and physical systems. 
We use the parsimonious \texttt{\textbf{N}onlinear} \texttt{\textbf{A}uto-\textbf{R}egressive \textbf{M}oving \textbf{A}verage with e\textbf{X}ogeneous input} (\texttt{NARMAX}) \citep{NARMAXOLS}. which has powerful yet simple parsimonious representation capability  on real systems. We first identified a suitable \texttt{NARMAX}  structure and model parameters, compute equilibrium points -- about which we linearized to a form of \eqref{eq:system}, before we estimate the \textit{robustly stabilizing and optimal control policy} for the mixed $\htwo/\hinf$-control problem. 

Suppose that an input-output data from a real system defined by \eqref{eq:nlnr} has been collected. Denote this as $D^N = \{z_1, \cdots, z_m, u_1, \cdots, u_p, w_1, \cdots, w_q\}$. Let the maximum lags in the input, disturbance, and output data be denoted by $n_u, \, n_w$, and $n_y$ respectively. We fit a polynomial \texttt{NARMAX}  model to $D^N$ with the  power-form $\ell$-degree polynomial,
\begin{align}
		&z(t) = \theta_0 + \sum_{i_1 = 1}^{n} \theta_{i_1} x_{i_1}(t) + \sum_{i_1 = 1}^{n} \sum_{i_2 = i_1}^{n}\theta_{i_1}\theta_{i_2}x_{i_1}(t)u_{i_2}(t)\cdots
		\nonumber \\
		&+\sum_{i_1 = 1}^{n} \cdots \sum_{i_\ell = i_{\ell-1}}^{n}\theta_{i_1} \cdots \theta_{i_\ell}\,x_{i_1}(t)\cdots x_{i_\ell}(t) + e(t)
	\label{eq:narmax}
\end{align}
whose parameters $\theta_{i_1 \ldots i_m}, m \in [1, l]$ are to be identified. The model structure has order $n = n_z+n_u + n_w + n_e$, where $n_e$ is the maximum order for a \textit{pseudo-random binary sequence} $e(k)$ that aids identification robustness. The state variables are explicitly, 
\begin{small}
	\begin{align}
		x_m(k) = \begin{cases}
			z(k-m), \quad 1 \le m \le n_z \\
			u(k-(m-n_z)), \quad n_z +1 \le m \le n_z + n_u \\
			w(k-(m-n_z-n_u)), \, n_z + n_u+1 \le m 
			\\ \qquad \qquad\qquad\qquad\qquad\qquad \le n_z + n_u + n_w\\
			e(k-(m-n+n_e)), \, n - n_e+1 \le m \le n. 
		\end{cases}
	\end{align}
\end{small}
Equation \eqref{eq:narmax} admits a linear regression model of the form $z(t) = \sum_{i=1}^M\phi_i(t)\theta_i + e(t)$ for $t = 1, \cdots, N$ and a process noise $e(t)$. Or in matrix form: $Z = \Phi \Theta + \Xi$, where $\Phi = [\phi_1 \cdots \phi_M]$ denotes the regression matrix, and $\Theta=[\theta_1, \cdots, \theta_M]$ are parameters to be learned, typically in a regression process. The solution to the least squares cost $\min_\Theta \|z - \Phi \Theta\|_2$ yields the parameter estimates $\Theta$ for the nonlinear model. 

We adopt the computationally efficient Householder transformation in transforming the information matrix, $\Phi^T \Phi$ into a well-conditioned QZ-matrix partition. Afterwards, we recover  \texttt{NARMAX} parameters by solving the resulting triangular system of linear equations in a least squares sense. 

Given that the nonlinear structure is unknown ahead of time, we start with large values of $n_z, n_u, \text{ and } n_w$ in $P$ -- adding regression variables that capture natural properties such as damping and friction e.t.c. in order to capture as many nonlinear variation that exist in the data as possible. We then iteratively pruned the parameters  using the \textit{error reduction ratio} algorithm~\citep{Billings2010} within the forward orthogonal regression least square algorithm~\citep{NARMAXOLS}.  

\subsection{NARMAX Linearization}
The conditions stipulated in Assumption \ref{ass:realizability} must be realized before we can implement a learning-based procedure. The identified \texttt{NARMAX} model is then linearized about a suitable equilibrium point to obtain a form of \eqref{eq:system} in state space form. In our experience, we have always found Assumption \ref{ass:realizability} to be satisfied after linearization. Suppose that the pair $(A,B)$ is still not controllable after linearization (we have not seen this in practice), a perturbation can be made of the multi-input LTI system as follows: Suppose that there exists a diagonalizable matrix $M$ such that $\bar{A} = M^\top A M$, and $\bar{B} = M^\top B$. Then $(A, B)$ can be reduced to $(\bar{A}, \bar{B})$ as follows:
\begin{align}
	\bar{A} = \begin{bmatrix}
		A_{cnt} & * \\ 0 & \bar{A}_{cnt}
	\end{bmatrix}, \quad  \bar{B} = \begin{bmatrix}B_{cnt} & \cdots & 0\end{bmatrix}^T
\end{align}
for\footnote{$\bar{X}_{cnt}$ signifies the non-controllable part of $X$.}
\begin{align}
	A_{cnt} = \begin{bmatrix}
		A_{11} & A_{12} \cdots A_{1,p-1} & A_{1p} \\
		A_{21} & A_{22} \cdots A_{2,p-1} & A_{2p} \\
		0 & A_{32} \cdots A_{3,p-1} &A_{3p} \\
		\vdots   & \vdots & \vdots \\
		0   & \cdots & A_{pp}
	\end{bmatrix}, \,\,
	B_{cnt} = \begin{bmatrix} B_1 \\ 0 \\ \vdots \\ 0\end{bmatrix}
\end{align}
where $p$ is a pair's controllability index, blocks $B_1, A_{21}, \linebreak[0] \cdots, A_{p,p-1} $ possess full row ranks, and $ dim(\bar{A}_{cnt}) = dim[\overline{(A, B)}_{cnt}$.

\subsection{LEQG/LQ Differential Game}
In~\citep[Algorithm II]{CuiMoluxTAC}, we introduced an iterative solver for the closed-loop controls $u$ and $w$ respectively. 
A key drawback is the need for the first $K_1 \in \mc{K}$ to be known. This limits the practicality of the algorithm to data-driven PO schemes. In addition, our examples did not illustrate a means of respecting the \textit{stabilizability} and \textit{detectability} assumptions needed to guarantee a solution to the minimax problem \eqref{eq:lq_minimax}. We now provide an all-encompassing learning scheme for obtaining the solution to the mixed $\htwo/\hinf$-control problem in a purely data-driven setting compatible with modern model-free policy optimization schemes. 

Let $p \in \mc{P}$ and $q \in \mc{Q}$ denote the iteration indices at which the controllers $u_p(t)=K_p x(t)$ and $w_q(t) = L_p^qx(t)$ are updated in \eqref{eq:disturb_control},  where $K_p$ and $L_p^q$ are feedback gains  
\begin{align}
	K_p = -R^{-1}B_1^T P_{p}^q, \,\, L_{p}^q = \gamma^{-2}B_2^TP_{p}^q.
	\label{eq:gains_cl}
\end{align}
Furthermore, let the closed-loop transition matrix under the gains of \eqref{eq:gains_cl}, and quadratic matrix term in \eqref{eq:lq_minimax} be (see \citep[Equation 12]{CuiMoluxTAC})
\begin{align}
	A_\gamma = A - B_1 K_p + B_2 L_p^q, \,\,Q_\gamma = C^TC + K_p^T R K_p.
	\label{eq:tran_mat_closed_loop}
\end{align}

\textbf{Observe}: $P_p^q$ is finite if and only if the closed-loop system matrix $A_\gamma$ possesses eigenvalues with negative real parts. In this case, $P_p^q, \text{ for } p, q = 0, 1, 2, \cdots $ is the unique positive definite solution to the closed-loop GARE 
\begin{align}
	&A_\gamma^T P_p^q + P_p^q A_\gamma + Q_\gamma - \gamma^{-2}L_p^{qT} L_p^q=0
	\label{eq:riccatti_closed_loop}
\end{align}
where recursively, \eqref{eq:gains_cl} holds for $p,q = 1, 2,$. Note that $K_1$ must be chosen such that $A_\gamma^1 = A - B_1 K_1 + B_2^T L_1^1$ is Hurwitz and its closed-loop $\hinf$ norm transfer function is bounded from above by a user-defined $\gamma >0$~\citep{Kleinman1968}. Then 
\begin{inparaenum}[(i)]
	\item $K_1 \le P_{p+1}^q \le P_p^q \le \cdots, \quad p,q = 1, 2,\, $ 
	\item $\lim_{(p,q)\rightarrow \infty} \, P_p^q = P$. See proof in~\citep{CuiMoluxTAC}.
\end{inparaenum}
%

\subsection{Mixed Sensitivity Initialization}
\begin{algorithm}[tb!]
	\caption{Search for the closed-loop $\hinf$-norm
		\label{alg:gamma_finder}}
	\begin{algorithmic}[1]
		\State{Given a user-defined step size $\eta>0$}
		\State Set the initial upper bound on $\gamma$ as $\gamma_{ub} = \infty$.
		\State Initialize a buffer for possible $\hinf$ norms for each $K_1$ to be found,  $\Gamma_{buf} = \{\}.$
		\State Initialize ordered \texttt{poles} $\mc{P} = \{p_i \in Re(s) <0 \,| \, i = 1,2, \}$ \Comment{$p_1  < p_2 <\cdots$}
		\For{$p_i \in \mc{P}$}
		\State Place $p_i$ on \eqref{eq:system};  \label{line:place}  \Comment{\Citep{TitsYang}}
		\State Compute stabilizing $K_1^{p_i}$ \label{line:gain_def}
		\State Find lower bound $\gamma_{lb}$ for $H(\gamma, K_1^{p_i})$; \Comment{using \eqref{eq:gamma_lb}}
		\State $\Gamma_{buf}(i)$ = \texttt{get\_hinf\_norm}($T_{zw}, \gamma_{lb}, \, K_1^{p_i}$).
		\EndFor
		\Function{\texttt{\textproc{get\_hinf\_norm}}}{$T_{zw}, \gamma_{lb}, \, K_1^{p_i}$}
		\While{$\gamma_{ub} = \infty$}
		\State $\gamma:= (1 + 2 \eta)\, \gamma_{lb}$;
		%
		\State Get $\lambda_i(H(\gamma, K_1^{p_i}))$ \Comment{\cf \eqref{eq:ham_cl}}
		\If{\texttt{Re}($\Lambda)\neq \emptyset$ for  $\Lambda=\{\lambda_1, \cdots \lambda_n\}$}
		\State Set $\gamma_{ub} = \gamma$;  exit
		\Else
			\State Set buffer $\Gamma_{lb} =\{\}$
			\For{$\lambda_k \in \{\texttt{Imag}{(\Lambda})_{:p-1}\}$} \Comment{$k=1$ to $K$}
				\State Set $m_k = \frac{1}{2}(\omega_k + \omega_{k+1})$	
				\State Set $\Gamma_{lb}(k) = \max \{\sigma\left[T_{zw}(jm_k)\right]\}$;
			\EndFor
			\State $\gamma_{lb} = \max(\Gamma_{lb})$
		\EndIf
			\State Set $\gamma_{ub} = \frac{1}{2}(\gamma_{lb} + \gamma_{ub})$.
		\EndWhile
			\State \Return $\gamma_{ub}$
		\EndFunction
	\end{algorithmic}
\end{algorithm} 
In~\citep{CuiMoluxTAC}, we had established that in order for our model-free algorithm to work in a purely data-driven setting, $K_1$ must be in the constraint set $\mc{K}$. The means for finding a $K_1$ that satisfies the constraints equation \eqref{eq:constraint} is itemized in algorithm \ref{alg:gamma_finder}. 

Following Corollary \ref{cor:gamma_inf}, we must first find the upper bound of $\gamma$ \ie $\gamma^\infty:=\gamma_{ub}$ whereupon the unique, finite, and positive definite solution to the GARE is satisfied. Let us now introduce the following proposition.
\begin{proposition}[\cite{HinfFastCompute}]
	For all $\omega_p \in \reline$, we have that $j\omega_p$ is an eigenvalue of the Hamiltonian $H(\gamma_1)$ if and only if $\gamma_1$ is a singular value of $T_{zw}(j\omega_p)$.
	\label{prop:steinbuch}
\end{proposition}
\begin{remark}
	Singular value computation is easily obtainable given a frequency of a system. Proposition \ref{prop:steinbuch} allows us to obtain all frequencies that correspond to a single eigenvalue. 
\end{remark}

\textbf{Procedure}: We search for stabilizing gains $K_1^{p_i}$ over the space of negative reals (or range matrix inequalities for multiple input systems) such that each $K_1^{p_i}$ on line \ref{line:gain_def} of Algorithm \ref{alg:gamma_finder} is stabilizing. A starting lower bound for each $\gamma_i$ corresponding to gain $K_1^{p_i}$ can be set 
\begin{align}
	\gamma_{lb}:=\max \{\eigmax(G(0)),\eigmax(G(j\omega_p)), \eigmax(B_2)\}
	\label{eq:gamma_lb}
\end{align}
where $\omega_p$ is chosen as specified in \citep[Eq. (4.4)]{HinfFastCompute}. The $\hinf$ computation scheme is fast and has a guaranteed quadratic convergence (since $\gamma_{lb}(i+1) > \gamma_{lb}(i)$) if the poles are chosen to make $(A-B_1K)$ Hurwitz. Given a user-defined step size, $\eta$, the rest of the algorithm consists in iteratively increasing the value of $\gamma_{lb}$ until all eigenvalues of the closed-loop system Hamiltonian \cf \eqref{eq:ham_cl} have no imaginary part.

The $\hinf$ norm computation scheme is based on the singular values of \eqref{eq:tranfer_cl} and the eigenvalues of \eqref{eq:ham_cl}. For poles far to the left of the origin, $\gamma$ will be small. However, as $\lambda \rightarrow 0$, $\gamma \rightarrow \infty$. Thus, a critical value of $\gamma^\star$ can be obtained (see Fig. \ref{fig:hinf}) above which the system becomes unstable (the cost becomes infinite; \cf\citep[Fig. 1]{LekaniDG}). We employ this heuristic to choose a $K_1$ that satisfies constraint \eqref{eq:constraint}.
\begin{figure}[tb!]
	\centering 
	\includegraphics[width=\columnwidth]{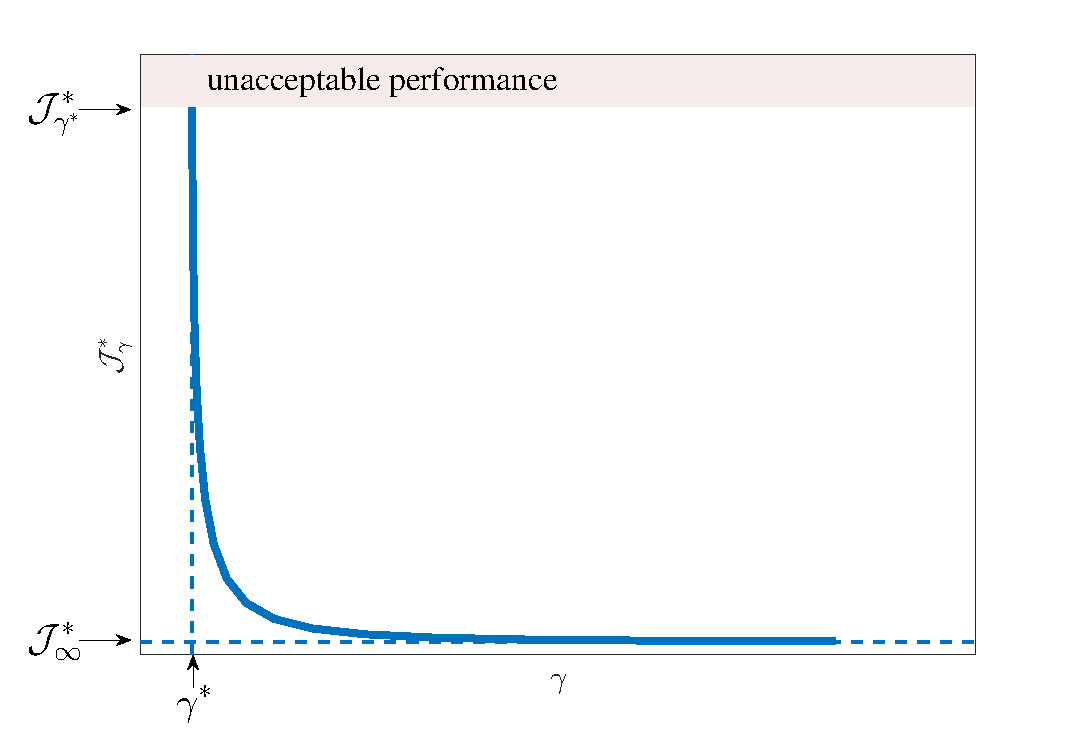}
	\caption{\footnotesize{Performance Robustness Trade-off Curve.}}
	\label{fig:perf_robust}
\end{figure}
\begin{figure}[tb!]
	\centering 
	\includegraphics[width=\columnwidth]{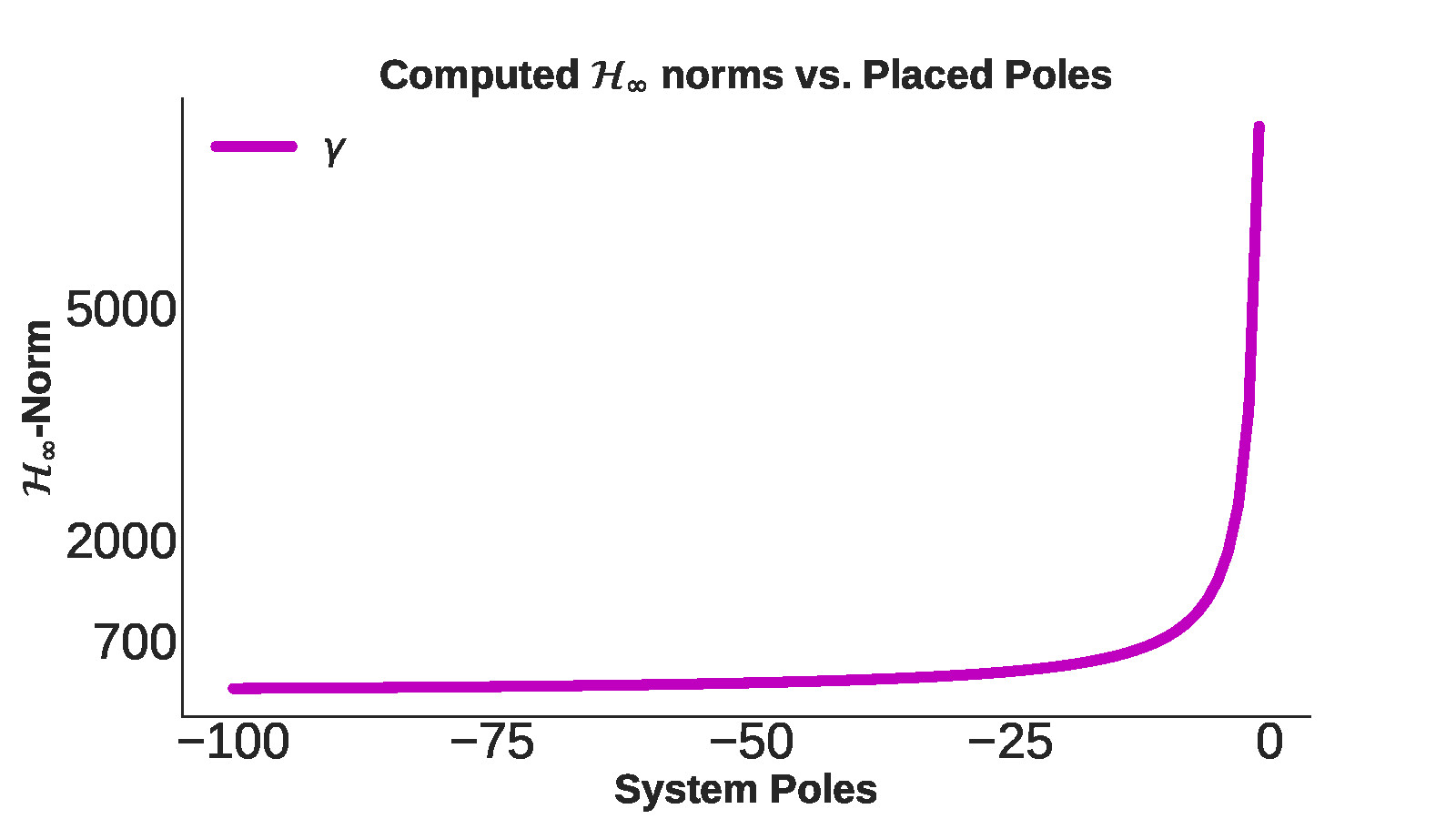}
	\caption{\footnotesize{Computed $\hinf$-norm for searched poles. We see that the closer to the origin, the greater the value of $\| T_{zw}\|_{\hinf}$}}
	\label{fig:hinf}
\end{figure}
This algorithm \ref{alg:gamma_finder} finds the system's $\hinf$ norm; part of it is an adaptation of~\citep{HinfFastCompute}'s fast $\hinf$ computation algorithm which enjoys quadratic convergence as opposed to the popular bisection algorithm~\citep{ZhouEssentials}. On line \ref{line:place}, we used~\citep{TitsYang}'s globally fast and convergent pole placement algorithm and leveraged the implementation provided in the \texttt{scipy} library.

\subsection{Iterative Two-Player LQ Zero-Sum Game}
We now analyze the dynamic game. Putting \eqref{eq:tran_mat_closed_loop} into  \eqref{eq:riccatti_closed_loop}, we have
\begin{align}
	(A^TP_p^q +& P_p^qA)+(Q_p-\gamma^{-2}L_p^{qT}L_p^q) + L_p^{qT} B_2^TP_p^q - \nonumber \\
	& - K_p^T B_1^T P_p^q -P_p^q B_1K_p + P_p^q B_2 L_p^q =0,
\end{align}
which in vector form can be written as 
%
\begin{align}
	&\svec\left(A^T P_p^q + P_p^q A\right) + \svec\left(Q_p - \gamma^{-2}L_p^{qT}L_p^q\right) 
	\label{eq:mf_vect} \nonumber \\
	&-\smat
	\left[(I_n\otimes K_p^T) + (K_p^T \otimes I_n)T_{\vec}\right]\vec(B_1^T P_p^q)     \\
	& +\smat\left[I_n \otimes L_p^{qT}B_2^T + (L_p^q B_2^T \otimes I_n)\right] \mat(\svec(P_p^q)) \nonumber 
\end{align}
where $p,\,q$ are iteration indices for the controller $u$ and disturbance $w$ respectively (introduced formally in Algorithm \ref{alg:mixed_synthesis}) and $n, m$ 
are as defined in \S \ref{subsec:sys}. We know that the differential game \eqref{eq:lq_minimax} admits equal upper and lower \textit{optimal} values owing to the GARE \eqref{eq:lq_two_player} having a positive definite solution i.e. $\mc{J}^\star = x^T P x$~\citep[Th 4.8 (iii)]{BasarBook}. Control laws must therefore be computed along the trajectories of \eqref{eq:system}, using the derivative of $\mc{J}^\star = x^TPx$. At the iteration pair $(p,q)$, $d(\mc{J}^\star(x; t)))$ admits the solution (by It\^{o}'s differential rule) 
\begin{align}
	\begin{split}
		&d(x^TP_p^qx) = x^T(A^TP_p^q + P_p^qA_p^q)xdt + 2x^T P_p^qB_1u_p dt \\
		&\qquad + 2x^T P_p^q B_2dw + \text{Tr}(B_2^T P_p^q B_2) dt
	\end{split}
	\label{eq:value_positive}
\end{align}
where $\text{Tr}(M)$ denotes the trace of $M$. 

Letting $\phi(t) = [\vecv^T(x), 2(x^T \otimes u^T), 1]^T$, and integrating the above on the interval $[0, t_f]$, we find that 
\begin{align}
	\begin{split}
		&\underbrace{\frac{1}{t_f}\int_{0}^{t_f} \phi d(\vecv^T(x))}_{\hat{\Psi}(t_f)}\svec(P_p^q) 
		= \begin{bmatrix}\svec({A^T P_p^q + P_p^qA}) \\  \vec(B^TP_p^q) \\ Tr(B_2^T {P} B_2)\end{bmatrix} \times \\
		& \quad\underbrace{\frac{1}{t_f}\int_{0}^{t_f} \phi \phi^T dt}_{\hat{\Phi}(t_f)} +  \frac{1}{t_f}\int_{0}^{t_f}2\phi x^T P_p^q B_2dw
	\end{split}
	\label{eq: intergration}
\end{align}
where the last term tends to zero as $t_f \rightarrow \infty$. We now recall Lemmas A.7 and A.8 in~\citep{CuiMoluxTAC}, so that the following holds almost surely: $ \lim_{t_f \to \infty} \hat{\Phi}(t_f) = \Phi \equiv \mathbb{E}(\phi \phi^{T})$, and $\Psi = \lim_{t_f\rightarrow \infty} \hat{\Psi}(t_f)$. Hence, 
\begin{align}
	\begin{bmatrix} 
		\svec(A^TP_p^q + P_p^qA) 
		\\  
		\vec(B_1^TP_p^q) 
		\\ 
		\text{Tr}(B_2^T P_p^q B_2)
	\end{bmatrix} = \Phi^{-1}(t_f) \Psi(t_f) \,\, \svec(P_p^q).
\end{align}
Furthermore, let $n_1:=n(n+1)/2$ and $n_2:=n_1+mn$. Then, we may write
\begin{subequations}
	\begin{align}
		\svec({A^T P_p^q + P_p^qA}) &= [\Phi^{-1}]_{[1:n_1]} \Psi \,\, \svec(P_p^q) \\
		\vec(B_1^TP_p^q) &= [\Phi^{-1}]_{[1+n_1:n2]} \Psi \,\, \svec(P_p^q).
		\label{eq:mf_split_2}
	\end{align}
	\label{eq:mf_split}
\end{subequations}
Define $\hat{\Phi}_1 = [\Phi^{-1}]_{[1:n_1]}$, $\hat{\Phi}_2 = [\Phi^{-1}]_{[1+n_1:n2]}$, so that \eqref{eq:mf_vect} in light of \eqref{eq:mf_split} becomes 
%
%
\begin{align}
	&\hat{\Phi}_1 \Psi \,\, \svec(P_p^q)+ \svec\left(Q_p - \gamma^{-2}L_p^{qT}L_p^q\right) 
	\nonumber \\
	&-\smat	\left[(I_n\otimes K_p^T) + (K_p^T \otimes I_n)T_{\vec}\right]\hat{\Phi}_2 \Psi \,\, \svec(P_p^q)   \nonumber \\
	& +\smat\left[I_n \otimes L_p^{qT}B_2^T + (L_p^q B_2^T \otimes I_n)\right] \mat(\svec(P_p^q))=0 
\end{align}
Rearranging the above and letting
\begin{align}
	\Upsilon_p^q &= \hat{\Phi}_1 \Psi-\smat	\left[(I_n\otimes K_p^T) + (K_p^T \otimes I_n)T_{\vec}\right]\hat{\Phi}_2 \Psi \nonumber \\
	+&\smat\left[I_n \otimes L_p^{qT}B_2^T + (L_p^q B_2^T \otimes I_n)\right] \mat(
	\label{eq:upsilon}
\end{align}
it can be verified that the cost matrix $P_p^q$ (\cf~\eqref{eq:disturb_control}) admits the solution
\begin{align}
	\svec(P_p^q) = (\Upsilon_p^q)^{-1}\svec\left(Q_p - \gamma^{-2}L_p^{qT}L_p^q\right).
	\label{eq:riccati_main}
\end{align}
The entire procedure for updating the control laws in an iterative manner is described in Algorithm \ref{alg:mixed_synthesis}. 
\begin{algorithm}
	\caption{Mixed $\mathcal{H}_2/\mathcal{H}_\infty$-Control Synthesis \label{alg:mixed_synthesis}}
	\begin{algorithmic}[1]
		\State Collect I/O data, and identify the nonlinear model \eqref{eq:nlnr} following  \S \ref{subsec:narmax}. \label{line:narmax_ident}
		\State Obtain $Z^L = \{A, B_1, B_2, C, D\}$ by linearizing the NARMAX model in step \ref{line:narmax_ident} \Comment{See \S \ref{subsec:narmax}}
		\State Form matrices $R:=D^TD \succ 0$, $Q:=C^TC$
		\State Using $Z_L$, find $\hat{K}_{1} \in \mathcal{K}$   \Comment{alg.  \ref{alg:gamma_finder}; Pick a suitable $\gamma$}
		\State Run \eqref{eq:system}'s time response with $K_1 \in \mc{K}$ on $Z^L$. 
		\State Form state-control data $(\hat{\Phi}_1, \, \hat{\Phi}_2, \hat{\Psi})$ \Comment{Eq.  \eqref{eq:upsilon}};\label{alg:mf_data_collect}
		\State $\hat{K}_{\bar{p}}^{\bar{q}}, \hat{L}_{\bar{p}}^{\bar{q}}, \hat{P}_{\bar{p}}^{\bar{q}}$ = \texttt{robust\_gains}($\hat{\Phi}_1, \, \hat{\Phi}_2, \hat{\Psi}, x(t), \bar{p}, \bar{q}, \gamma$)
		\For{$t=1,\cdots, t_f$}
			\State Apply $u(x) = \hat{K}_{\bar{p}}^{\bar{q}} x(t), \, w(x) = \hat{L}_{\bar{p}}^{\bar{q}} x(t)$
			to eq. \eqref{eq:system} \label{alg:mf_time_step_control}
		\EndFor
		\Function{\texttt{robust\_gains}}{$\hat{\Phi}_1, \, \hat{\Phi}_2, \hat{\Psi}, \bar{p}, \bar{q}, \gamma$}
			\State Initialize $q=1$ 
			\State Initialize $L_1^1$ \Comment{Set to zeros or randomly initialize} 
			\For{$p \in 1 \text{ to } \bar{p}$}
				\While{$q \leq \bar{q}$}
					\State Find in order: $\hat{\Upsilon}_p^q, \hat{P}_{p}^q$ \Comment{Eq.  \eqref{eq:upsilon} \& \eqref{eq:riccati_main}};
					\State Compute $\hat{L}_{p}^q\leftarrow \gamma^{-2} B_2^T \hat{P}_{p}^q$ \Comment{Eq. \eqref{eq:disturb_control}}
					\State Update $q \leftarrow q+1$
				\EndWhile
			\State Form $\smat(\vec(\widehat{B_1^T{P}_{p}^{\bar{q}}}))$ 
			\Comment{Eq. \eqref{eq:mf_split_2}}
			\State Compute $\hat{K}_{p+1}^{\bar{q}} \leftarrow R^{-1} \widehat{B_1^T \, {P}_{p}^{\bar{q}}}$ \Comment{Feedback gain};
			\EndFor
			\State \Return $\hat{K}_{\bar{p}}^{\bar{q}}, \hat{L}_{\bar{p}}^{\bar{q}}, \hat{P}_{\bar{p}}^{\bar{q}}$
		\EndFunction
	\end{algorithmic}
\end{algorithm} 

We note in passing that inversion of matrix terms are efficiently computed using standard Cholesky factorizations. We refer interested readers to~\citep{CuiMoluxTAC} for the convergence and robustness analyses of our results.
\section{Results}
\label{sec:results}

We now present numerical results of the algorithm described in the foregoing.
\subsection{Car Cruise Control System}
We consider a car \textit{cruise control} system \citep[\S 3.1]{FdbkSys21} whereupon \textit{a controller $u(x(t)) = [u_1(t), u_2(t)]$ must maintain a constant velocity $v$ (the state), whilst automatically adjusting the car's throttle, $u_1(t), t \in [0, T]$} despite disturbances characterized by road  slope changes ($u_3 = \theta$), rolling friction ($F_r$), and aerodynamic drag forces ($F_d$). 

This control design problem is well-suited to our robust control formulation because
\begin{inparaenum}[(i)]
	\item the disturbances and state variables are separable and can be lumped into the form of the stochastic differential equations \eqref{eq:nlnr} and \eqref{eq:system};
	\item it is a multiple-input (throttle, gear, vehicle speed) single-output (vehicle acceleration) system that introduces modeling challenges;
	\item the entire operating range of the system is nonlinear though there is a reasonable linear bandwidth that characterize the input/output (I/O) system as we will see shortly.
\end{inparaenum}
The model is
%
	\begin{align}
		m \dfrac{dv}{dt} &= \alpha_n u \tau (\alpha_n v) - m g C_r sgn(u)  \nonumber \\ &\qquad - \dfrac{1}{2} \rho C_d A |v| v - mg \sin \theta
		\label{eq:data_collect}
	\end{align}
%
where  $v$ is the velocity profile of the vehicle (taken as the system's state), $m$ is vehicle's mass, $\alpha_n$ is the inverse of the vehicle's effective wheel radius, $\tau$ is the vehicle's torque -- it is controlled by the throttle $u:=u_1$. The rolling friction coefficient is  $C_r$ and $C_d$ is the aerodynamic drag constant for a vehicle of area $A$. The road curvature, $\theta$, is modeled as a Wiener process \cf \eqref{eq:system} with $w_i(t) \sim \mc{N}(0, 1)$ where for $i \in [1,\cdots,]$, $dw_i = \sum_{j=1}^i w_j$. If we let $x := v, \, u_1 := u, \, u_2:= \alpha_n$, and $u_3 := \theta$ 
%
%
and set $C_r = 0.01$, $C_d = 0.32$, $\rho = 1.3kg/m^{3}$, $A=2.4 m^2$ (following~\citep{FdbkSys21}), then the torque $\tau$ is $\tau = \tau_m-\tau_m\beta\left({\omega}/\omega_m-1\right)^2
$, where $\beta = 0.4, \, \omega_m = 420$ and $\tau_m = 190$. Simplified, we write 
$$\tau  = 190 - 76 \left(\dfrac{39 x}{420}-1\right)^2.$$
\subsection{Nonlinear Identification and Linearization}
\label{subsec:res_linearization}
In our \texttt{NARMAX} structure selection and model estimation scheme, we first start with a large number of parameters and regressors that consists of the polynomial expansion in \eqref{eq:narmax}, \texttt{sinuoisal}, and \texttt{signum} functions following \eqref{eq:data_collect}. We choose a polynomial degree of $3$ and the inputs $u$ and state lags $x$ were chosen as $[1,1,1]$ and $[1]$ respectively. We employed the forward regression orthogonal least squares algorithm~\citep{NARMAXOLS} in estimating the parametric terms of the model. We then employed the  \texttt{error reduction ratio}~\citep{NARMAXOLS, Billings2010} algorithm in pruning away extraneous terms. This whittled down the eventual model to the following parsimonious representation
\begin{align}
	\dot{x}(t) &= 0.062518 u_2(t) x(t)-0.12051 u_1(t)u_2^2(t)   \nonumber \\
	& + 0.00081339 u_2^3 (t) + 0.9767 \sin (u_3(t)).
\end{align}

The \texttt{NARMAX} structure selection and model estimation step produced  a \texttt{root relative test error} of $0.0661$ (see Fig. \ref{fig:ident_pred}). Using \eqref{eq:data_collect} with $u_3 \sim \mc{N}(0, 0.05)$, a constant gear ratio of $40$ and a car mass of $1600kg$, we collect I/O data with 40,000 samples in continuous time as shown in Fig. \ref{fig:ident_io}.With the input-output data, a NARMAX model was identified whose prediction error with respect to held-out validation data is shown in Fig. \ref{fig:ident_pred}. 
%
\begin{figure}[tb!]
	\centering
	\begin{minipage}[tb]{0.95\columnwidth}
		\includegraphics[width=\textwidth]{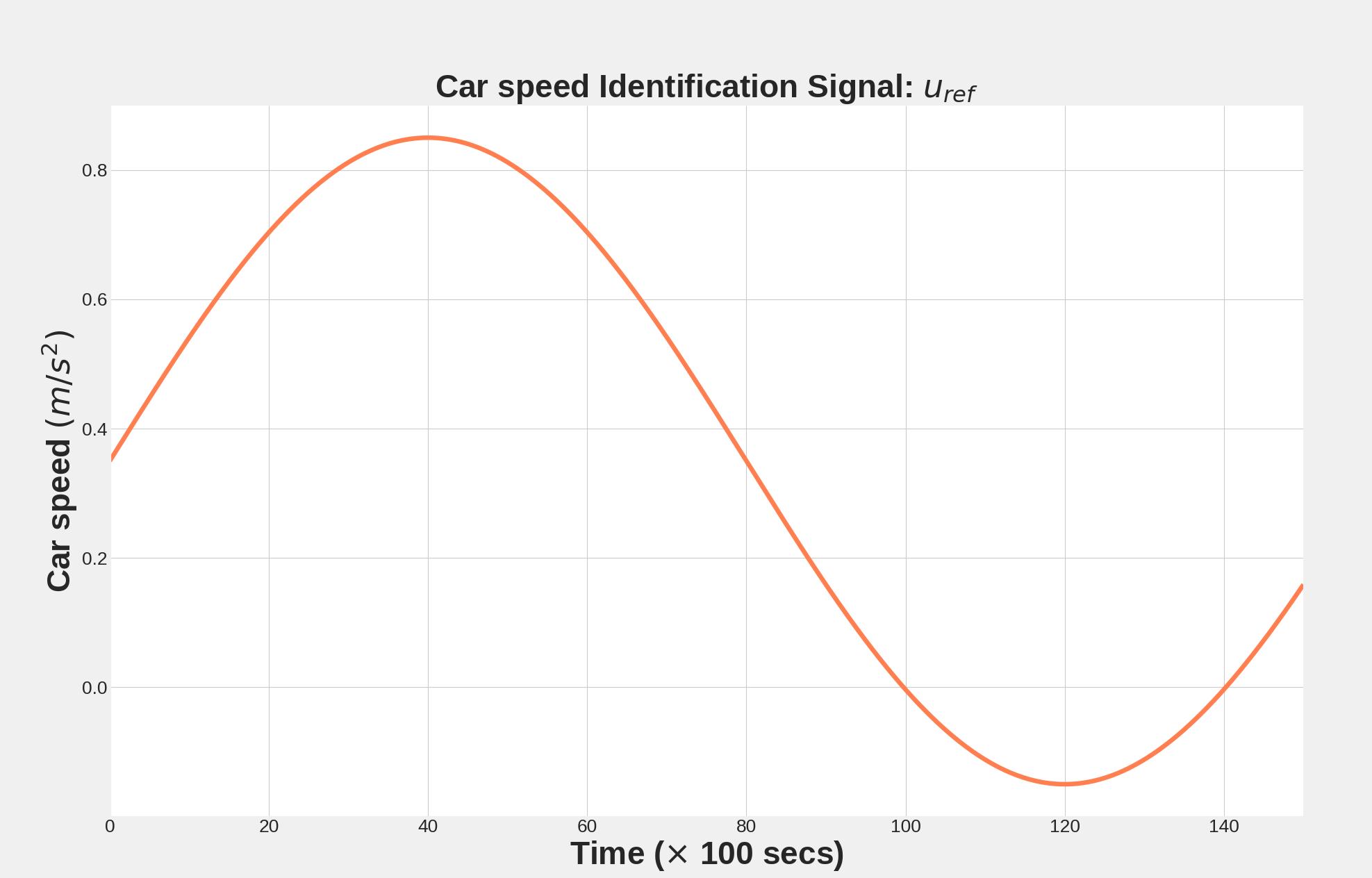}
	\end{minipage}
	\begin{minipage}[tb]{0.95\columnwidth}
		\includegraphics[width=\textwidth]{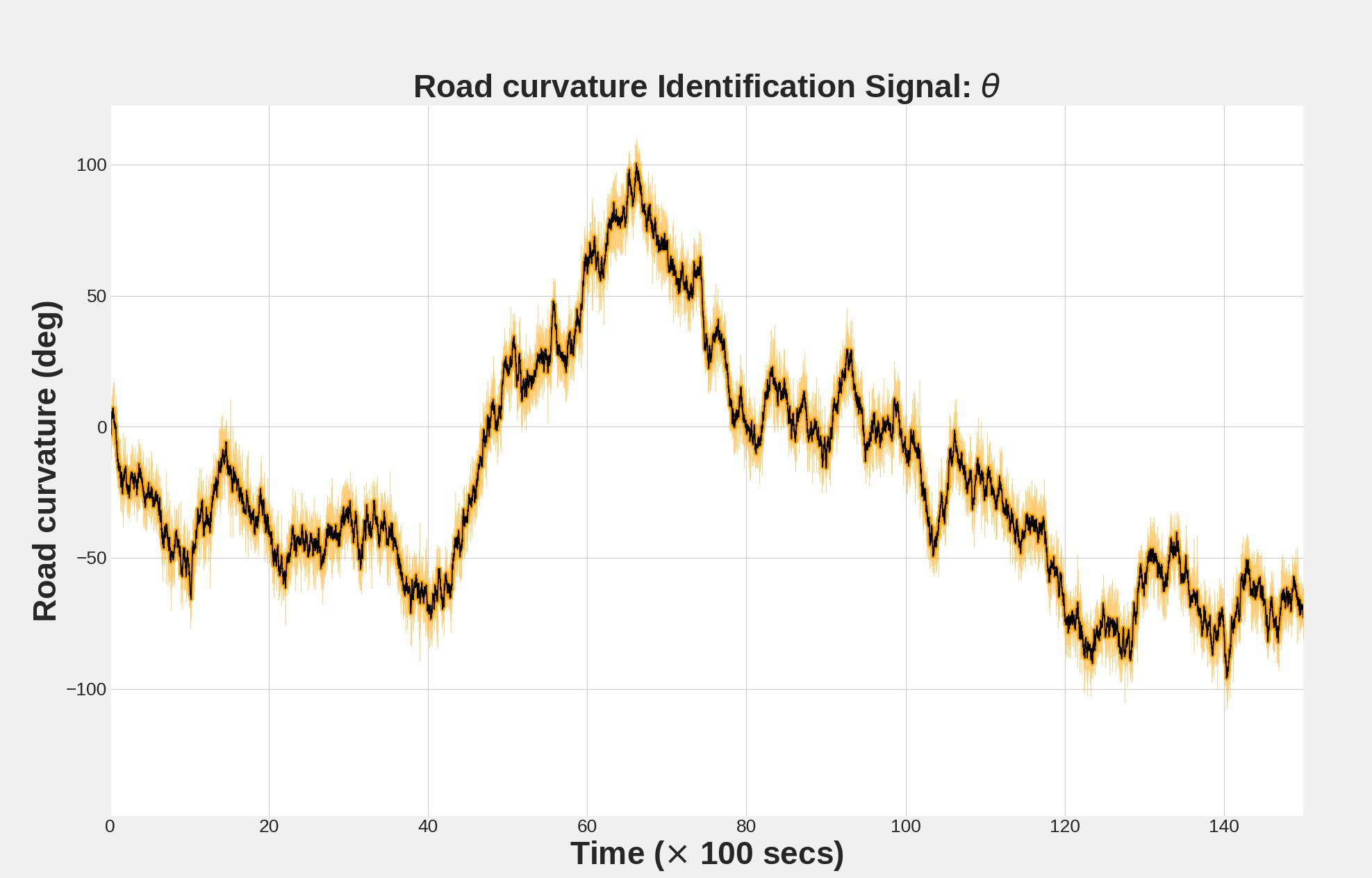}
	\end{minipage}
	\caption{\centering{\footnotesize{Car speed and road inclination input identification signals.}}}
	\label{fig:ident_io}
\end{figure}

%
%
\begin{figure}[tb!]
	\centering
	\begin{minipage}[tb]{0.95\columnwidth}
		\includegraphics[width=\textwidth]{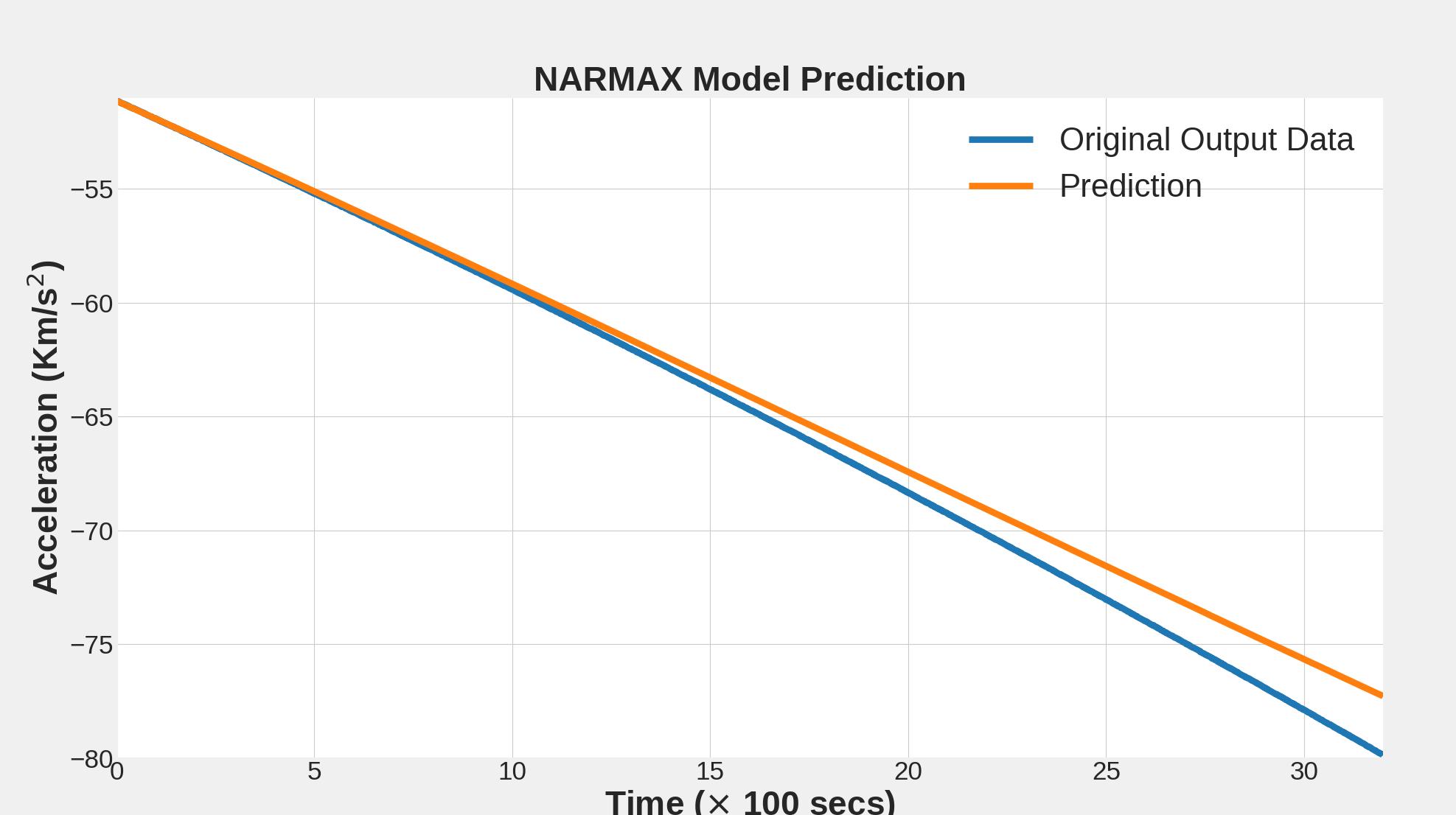}
	\end{minipage}
	\begin{minipage}[tb]{0.95\columnwidth}
		\includegraphics[width=\textwidth]{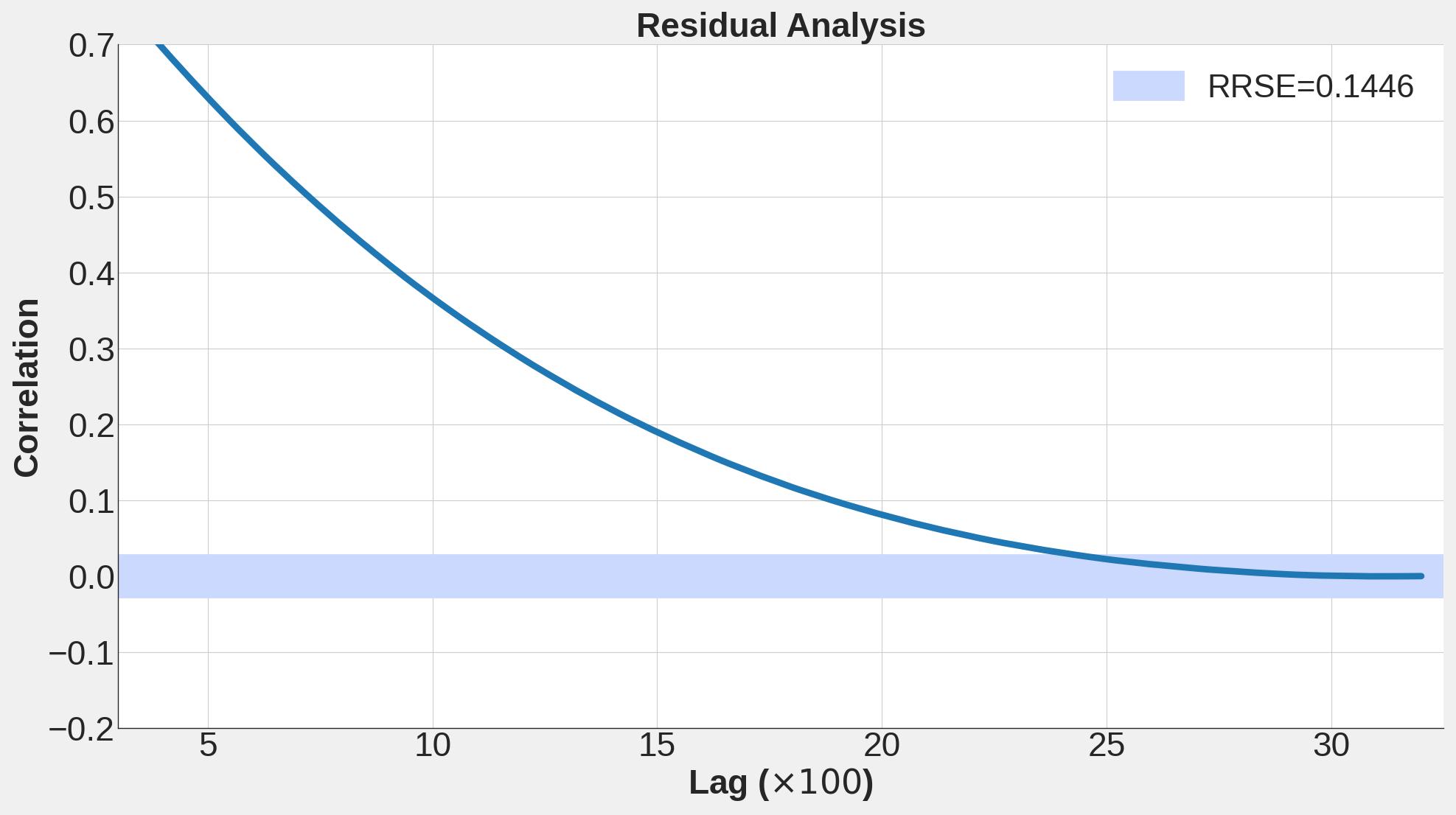}
	\end{minipage}
	\caption{\centering{\footnotesize{Top: Model prediction vs. actual car acceleration. Bottom: Residuals.}}}
	\label{fig:ident_pred}
\end{figure}

%
To amend the nonlinear control problem~\eqref{eq:nlnr} to the setup \eqref{eq:system}, we compute the values of the states, inputs, and outputs for system \eqref{eq:narmax}'s equilibrium points: $x_{eq}:=v_{ref}$, $u_{eq}:=\left[u_1, u_2 \right]$, and $z_{eq} \triangleq v_{ref}$ given initial  values $x(0) = 20$, $u(0) =  \left[0, 40\right]$, and $z(0) = 20$. The resulting linearized system \eqref{eq:system} is
\begin{align}
	A &= \left[\begin{array}{c}
		100.0288
	\end{array}\right], \,\, B_1 = \left[\begin{array}{cc}
	-193.072,  &  137.3123
\end{array}\right], \nonumber \\
	B_2 &= \left[\begin{array}{cc}
		-17014.7221,
		& -10557.48189
	\end{array}\right]\, C = \left[\begin{array}{cc}
1,  &  0
\end{array}\right]
\label{eq:linearized}
\end{align}
and $D$ is $[1,1]$. As seen, the pairs $(A, B_1)$ is stabilizable and $(C, A)$ is observable -- notable features of linearizing the \texttt{NARMAX} model in that it faithfully captures a system's parsimonious model.

\subsection{Efficacy of the Learning Algorithm}
After running Algorithm \ref{alg:gamma_finder}, we found  a $\gamma$ of value $500$ to be a suitable value for robustly compensating for a change in road slope with angle $\theta=40$. The goal is to regulate the speed of the car so that despite the change in slope, a constant speed of $40m/s$ is maintained. We then run Alg. \ref{alg:mixed_synthesis} for $R=\identity$, solve for $\Upsilon$ and equations \eqref{eq:upsilon} and \eqref{eq:riccati_main} based on collected data on the linearized system \eqref{eq:linearized}.
\begin{figure*}[tb]
	\centering
		\includegraphics [width=0.95\linewidth]{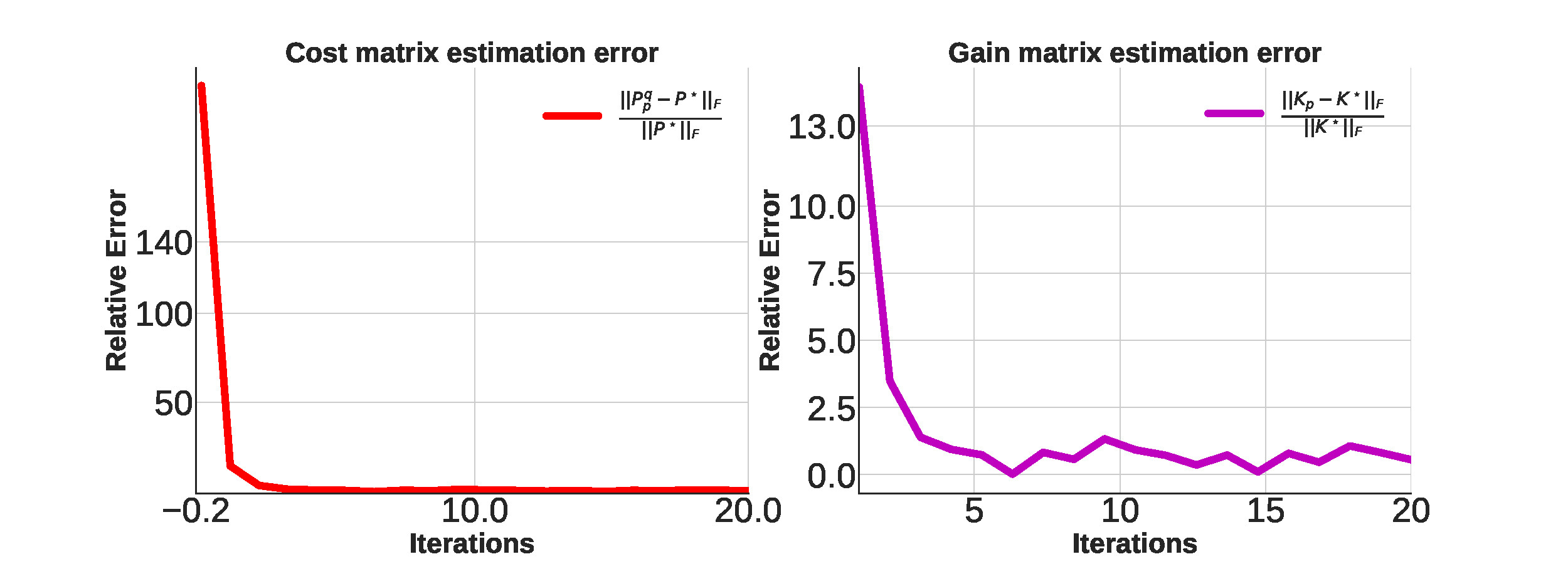}
	\caption{\footnotesize{Relative estimation error for the cost and gain matrices.}}
	\label{fig:gain_cost_mat}
\end{figure*}
We run Alg. \ref{alg:mixed_synthesis} on the collected data (see line \ref{alg:mf_data_collect} of Alg. \ref{alg:mixed_synthesis}). We then test the efficacy of the computed solutions to the final gains $K_{\bar{p}}^{\bar{q}}$ and cost matrix $P_{\bar{p}}^{\bar{q}}$ using our iterative solver (\cf Alg. \ref{alg:mixed_synthesis}) against (known) computed optimal values for the cost matrix $P^\star$ and gain $K^\star$ using Duncan's Riccati equation \eqref{eq:closed_loop_riccati} and control law in \eqref{eq:duncan_control}.  We set iteration max indices to $20$ and $30$. The relative errors between our solver and these solutions are shown in Fig. \ref{fig:gain_cost_mat}. We see that both parameters converge to their optimal values in within the first two iterations; this is despite the disturbance and unknown model aforetime.


\section{Conclusion}
\label{sec:conclude}
Following up on our recent contribution~\citep{CuiMoluxTAC}, we have  presented a framework for discarding the restricting assumption of \textit{stabilizability} and \textit{observability} in linear plants under a mixed sensitivity design framework. We first identified the nonlinear system, linearized it, found appropriate $\hinf$ bounds for the system and then deployed our learning algorithm. We introduced a fast means for finding the linearized system's $\hinf$-norm; and then simplified the two-loop mixed sensitivity algorithm earlier disseminated in~\cite{CuiMoluxTAC}. Further numerical results on a car's cruise controller is here presented to fortify the credibility of our previous results. Some open problems, which we intend to treat in the near future include 
 \begin{enumerate}
	     \item systems with finite-escape time in the solution to differential equations \eqref{eq:nlnr} and \eqref{eq:system} \ie there exists discontinuity that incapacitates the Lipschitz continuity assumption of $f(x, u)$ in $x$. Of what relevance is finite-time stability in such mixed sensitivity analyses?
	     \item multiplicative noise in the system dynamics; and
	     \item a large scale study of robustness analyses to distributed computing systems.
\end{enumerate}

\bibliography{ifac}

\end{document}